 \documentclass[12pt]{amsart}
\usepackage{amsmath}
\usepackage{amssymb}
\usepackage{graphicx}
\usepackage[colorlinks=true, urlcolor=blue,bookmarks=true,bookmarksopen=true,
citecolor=blue]{hyperref}

\numberwithin{equation}{section}

\newtheorem{theorem}{Theorem}[section]

\newtheorem{conjecture}[theorem]{Conjecture}
\newtheorem{corollary}[theorem]{Corollary}

\newtheorem{prop}[theorem]{Proposition}

\def \begineq{\begin{equation}}
\def \endeq{\end{equation}}

\def \bb{\mathbb}

\def \NN{{\bb{N}}}

\def \RR{{\bb{R}}}

\def \ZZ{{\bb{Z}}}

\def \({\left(}
\def \){\right)}
\def \<{\langle}
\def \>{\rangle}

\def \bar{\overline}
\def \deg{\mathrm{deg}}

\def \eps{\epsilon}

\def \om{{\omega}}
\def \Om{{\Omega}}
\def \Cc{\mathcal{C}}
\def \Si{{\Sigma}}
\def \ad{{\rm ad}}

\def \Aut{{\rm Aut}}

\def \End{{\rm End}}
\def \Ham{{\rm Ham}}
\def \Hom{{\rm Hom}}
\def \id{{\rm id}}

\def \tr{{\rm tr}}

\usepackage[applemac]{inputenc}
\def\scaledpicture#1by#2(#3scaled#4){{
\dimen0=#1  \dimen1=#2
\divide\dimen0 by 1000 \multiply\dimen0 by #4
\divide\dimen1 by 1000 \multiply\dimen1 by #4
\picture \dimen0 by \dimen1 (#3 scaled #4)}}
\def\dfigure#1by#2(#3scaled#4offset#5:#6)
    {\medskip
     \vglue 2mm minus 2mm
     $$
       \hbox{
         \hglue#5
         {\scaledpicture #1 by #2 (#3 scaled #4)}
       }
     $$
     \par\goodbreak
     \vglue 2mm minus 2mm
     \medskip}

\def\resto#1#2{{
#1\hskip 0.4ex\vline_{\hskip 0.2ex\raisebox{-0,2ex}
{{${\scriptstyle #2}$}}}}}

\def\N{{\mathbb N}}

\def\R{{\mathbb R}}
\def\Z{{\mathbb Z}}

\def\qed {\hfill\vrule height6pt width6pt depth0pt \smallskip}
\def\textmap#1{\mathop{\vbox{\ialign{
                                  ##\crcr
      ${\scriptstyle\hfil\;\;#1\;\;\hfil}$\crcr
      \noalign{\kern 1pt\nointerlineskip}
      \rightarrowfill\crcr}}\;}}
\def\bigtextmap#1{\mathop{\vbox{\ialign{
                                  ##\crcr
      ${\hfil\;\;#1\;\;\hfil}$\crcr
      \noalign{\kern 1pt\nointerlineskip}
      \rightarrowfill\crcr}}\;}}
      
\newcommand{\Cal}{\cal}
\newcommand{\cal}{\mathcal}
\def\textlmap#1{\mathop{\vbox{\ialign{
                                  ##\crcr
      ${\scriptstyle\hfil\;\;#1\;\;\hfil}$\crcr
      \noalign{\kern-1pt\nointerlineskip}
      \leftarrowfill\crcr}}\;}}

\def\fg{{\mathfrak f}}
\def\g{{\mathfrak g}}

\newtheorem{sz}{Satz}[section]
\newtheorem{thry}[sz]{Theorem}

\newtheorem{re}[sz]{Remark}
\newtheorem{co}[sz]{Corollary}

\newtheorem{lm}[sz]{Lemma}

\begin{document}
 
\def\tr{\mathrm {Tr}}
\def\End{\mathrm {End}}
\def\Aut{\mathrm {Aut}}
\def\Spin{\mathrm {Spin}}
\def\U{\mathrm{U}}
\def\SU{\mathrm {SU}}
\def\SO{\mathrm {SO}}
\def\PU{\mathrm {PU}}
\def\GL{\mathrm {GL}}
\def\spin{\mathrm {spin}}
\def\u{\mathrm {u}}
\def\su{\mathrm {su}}
\def\so{\mathrm {so}}
\def\pu{\mathrm {pu}}
\def\Pic{\mathrm {Pic}}
\def\Iso{\mathrm {Iso}}
\def\NS{\mathrm{NS}}
\def\deg{\mathrm {deg}}
\def\Hom{\mathrm{Hom}}
\def\Herm{\mathrm{Herm}}
\def\Vol{{\rm Vol}}
\def\pf{{\bf Proof: }}
\def\id{ \mathrm{id}}
\def\Im{\mathrm{Im}}
\def\im{\mathrm{im}}
\def\rk{\mathrm {rk}}
\def\ad{\mathrm {ad}}
%\def\coker{\mathrm{coker}}
%%%%%%%%%%%%%%%%%%%%%%%%%%
\def\spc{\mathrm{Spin}^c}
\def\U2{\mathrm{U(2)}}
\def\niq{=\kern-.18cm /\kern.08cm}
%%%%%%%%%%%%%%%%%%%
\def\Ad{\mathrm {Ad}}
\def\RSU{\R\mathrm{SU}}
\def\ad{{\rm ad}}
\def\dva{\bar\partial_A}
\def\da{\partial_A}
\def\p{{\rm p}}
\def\sp{\Sigma^{+}}
\def\sm{\Sigma^{-}}
\def\spm{\Sigma^{\pm}}
\def\smp{\Sigma^{\mp}}
\def\oo{{\scriptstyle{\cal O}}}
\def\ooo{{\scriptscriptstyle{\cal O}}}
\def\sw{Seiberg-Witten }
\def\pa{\partial_A\bar\partial_A}
\def\Dr{{\raisebox{0.17ex}{$\not$}}{\hskip -1pt {D}}}
\def\gr{{\scriptscriptstyle|}\hskip -4pt{\g}}
\def\subsetint{{\  {\subset}\hskip -2.45mm{\raisebox{.28ex}
{$\scriptscriptstyle\subset$}}\ }}
%\def\subsetint{{\  {\subset}\hskip -2.2mm{\raisebox{.28ex}
%{$\scriptscriptstyle\subset$}}\  }}

% added by RÃ©mi.
\def\nat{\Psi_L}
\def\comp{\mathrm{comp}}
\def\cm{MC} \def\cf{FC} \def\hm{MH} \def\hf{FH}
\def\pss{PSS}
\def\co{\colon\thinspace}

\def \qed{\hfill $\square$ \vspace{0.03in}}

\newcommand\rlabel[1]{\label{#1}\marginpar{{\it #1}}}
%\begin{document}

%%%%%%%%%%%%%%%%%%%%%%%%%%%%%%
%   TEXT FORMATTING

%%%%%%%%%%%%%%%%%%%%%%%%%%

%%%%%%%%%%%%%%%%%%%%%%%%%%

%%%%%%%%%%%           BEGINNING OF  TEXT

%%%%%%%%%%%%%%%%%%%%%%%%%%

%\begin{center}{\huge PRELIMINARY VERSION} \end{center}

\title{The $g$-areas and the commutator length}

\author[Fran\c{c}ois Lalonde]{Fran\c{c}ois Lalonde}
\address{D\'epartement de Math\'ematiques et de Statistique, Universit\'e de Montr\'eal, C.P. 6128 Succ. Centre-ville,
Montr\'eal (Qu\'ebec) H3C 3J7, 
Canada  }
\email{lalonde@dms.umontreal.ca}

\author[Andrei Teleman]{Andrei Teleman}
\address{UMR-CNRS 6632 (LATP)
CMI, Aix-Marseille Universit\'e,
13453 Marseille Cedex 13, 
France
 }
\email{andrei.teleman@univ-amu.fr }

\date{\today}


\bigskip

\begin{abstract} The commutator length of a Hamiltonian diffeomorphism $f\in \Ham(M, \om)$ of  a closed symplectic manifold $(M,\omega)$ is by definition the minimal $k$ such that $f$ can be written as a product of $k$ commutators in $\Ham(M, \om)$.  We introduce a new invariant for Hamiltonian diffeomorphisms, called the $k_+$-area, which measures the ``distance", in a  certain sense, to the subspace ${\cal C}_k$  of  all products of  $k$ commutators. Therefore this invariant can be seen as the obstruction to writing a given  Hamiltonian diffeomorphism as a product of $k$ commutators.
We also consider an infinitesimal version of the commutator problem: what is the obstruction to writing a Hamiltonian vector field as a linear combination of $k$ Lie brackets of Hamiltonian vector fields? A natural problem related to this question is to describe  explicitly, for every fixed $k$, the set of linear combinations of $k$ such Lie brackets. The problem can be obviously reformulated in terms of Hamiltonians and Poisson brackets. For a given Morse function  $f$ on a symplectic Riemann surface $M$ (verifying a weak genericity condition) we describe the linear space of  commutators of the form $\{f,g\}$, with $g\in{\cal C}^\infty(M,\R)$.
\end{abstract}

\maketitle

\section{Introduction}

    Let $(M, \om)$ be a closed symplectic manifold and $\Si_g$ the real oriented compact  surface with boundary obtained by removing an open disc from the real compact surface without boundary of genus $g$. So, in other terms, $\Si_g$ is obtained by attaching $g$ handles to the closed unit disk. Now let $(M, \om) \hookrightarrow (P, \Om) \stackrel{\pi}{\to} \Si_g$ be a Hamiltonian fibration with fiber $(M, \om)$ and base $\Si_g$. This means that $P$ is a smooth fibration with fiber $M$ over $\Si_g$ and structure group the group of Hamiltonian diffeomorphisms $\Ham(M,\omega)$ of $M$, and the connection defined by $\Omega$ is compatible  with this structure group (see \cite{LM}). This is equivalent to asking that the symplectic form $\Om$ on $P$ be ruled (which means that $\Om$ restricts on each fiber to a symplectic form isotopic to $\om$) with monodromy around any loop of the base belonging to $\Ham(M, \om)$. 
    
    Recall that $\Ham(M, \om)$ is simple, so the subgroup generated by all products of commutators in $\Ham(M, \om)$ must be the whole group, which means that every Hamiltonian diffeomorphism of $M$ can be written as a product of commutators.  The {\it commutator length} of $f \in \Ham(M,\om)$ is by definition the smallest $k$ such that $f$ can be written as a product $[\phi_1, \psi_1] \ldots [\phi_k, \psi_k] $ of $k$ commutators.
    
    Now let $(P, \Om) \stackrel{\pi}{\to} \Si_g$ be a ruled symplectic manifold with Hamiltonian monodromies. Because $\Om$ is non-degenerate on the fibers, the kernel  $K$ of the restriction $\Om |_{\partial \Si_g}$ of $\Om$ to the boundary $\partial \Si_g \simeq S^1$ is transversal to the fibers of the projection $ \pi^{-1}(\partial \Si_g) \to \partial \Si_g$. If $b_0$ is a base point belonging to $\partial \Si_g$, the monodromy of $K$ round $\partial \Si_g$ defines a symplectic diffeomorphism of $(M_{b_0}, \om_{b_0})$ which is Hamiltonian by hypothesis. Conversely, any Hamiltonian diffeomorphism $f$ of $(M, \om) \simeq (M_{b_0}, \om_{b_0})$ can be obtained as monodromy of a ruled symplectic fibration $(M, \om) \hookrightarrow (P, \Om) \stackrel{\pi}{\to} \Si_g$. This can be easily seen by first representing $f$ as monodromy of a ruled $(M, \om)$-fibration $(P, \Om)$ over $\Si_0$, then by trivialising $(P, \Om)$ over a small neighbourhood $U$ of an interior point $b$ of $\Si_0$ and by replacing finally $P |_U$ by $(M, \om) \times \Si_g$.
    
    We define the $g_+$-{\em area} $\| f \|_g^+$ as the infimum of the quotient
    $$
    \frac{vol (P, \Om)}{vol (M, \om)}
    $$
    when $(P, \Om)$  runs over all Hamiltonian fibrations
    $$
    (M, \om) \hookrightarrow (P, \Om) \stackrel{\pi}{\to} \Si_g
    $$
    whose monodromies round $\partial \Si_g$ equal $f$. Similarly, we define $\| f \|_g^-$ as $\| f^{-1} \|_g^+$, which means  filling $P |_{\partial \Si_g}$ from the other side with a copy of $\Sigma_g$. Finally, we set:
    $$
        \| f \|_g = \| f \|_g^+  + \| f \|_g^-.   
    $$    
        It is easy to see that $ \| f \|_0 $ is less or equal to the positive Hofer norm of $f$ (see \cite{H, LM1} for details on the Hofer norm). 
        
        The first result of this article is the following:
        
        \begin{theorem} \label{thm:theorem1} If $f$ is a product of $k$ commutators $ f = [\phi_1, \psi_1] \ldots [\phi_k, \psi_k]$, then
        $$
        \| f \|_g^+ = 0 
        $$
        for all $g \geq k$.
        \end{theorem}
        
           Actually, we will prove more. Let us denote by $\Cc_k$ the subset of $\Ham(M, \om)$ of all products of $k$ commutators of Hamiltonian diffeomorphisms. By convention, $\Cc_0$ contains only the identity. Then obviously $\Cc_0 \subset \Cc_1 \subset \ldots$ , $\cup_{k \geq 0} \Cc_k = \Ham(M,\omega)$ and the following theorem holds:
           
           \begin{theorem} \label{thm:theorem2} For all $f \in \Ham(M,\omega)$,
           $$
           \|f \|_k^+  \leq  d_0^+ (f, \Cc_k)
           $$
           where $d_0^+ (f, \Cc_k)$ is the infimum over $h \in \Cc_k$ of $d_0^+(f,h)$ and where $d_0^+(f,h)$ is $\|fh^{-1} \|_0^+$.
           \end{theorem}
           
           The first theorem is clearly a corollary of the second one. It will therefore be enough to establish the latter one.  Note that the positive 0-area appearing in the right hand side of the above theorem is smaller or equal to the positive Hofer norm $\| \cdot \|_H^+$. Thus we get the following corollary that relates our $g$-area to the positive Hofer norm :
           
           \begin{corollary}  For all $f \in \Ham(M,\omega)$,
           $$
           \|f \|_k^+  \leq  d_H^+ (f, \Cc_k)
           $$
           where $d_H^+ (f, \Cc_k)$ is the infimum over $h \in \Cc_k$ of $d_H^+(f,h)$ and where $d_H^+(f,h)$ is $\|fh^{-1} \|_H^+$.
\end{corollary}
           
           It is clear that the positive g-area is far from being a norm on the group of Hamiltonian diffeomorphisms. However, the natural object to consider is the following polynomial, that we call the {\em positive total Hofer norm} defined by:
           $$
           \|f \|_T^+ = \sum_{k \geq 0}  \|f \|_k^+  t^k.
           $$
           We define in a similar way the total negative norm. Let us denote by $\star$ the operation on polynomials obtained from the ordinary product by replacing in each coefficient the sum by the minimum and the product by the sum. Explicitely, set
           $$
           (\sum a_k t ^k) \star (\sum b_k t ^k) = \sum c_k t ^k
           $$
           where 
           $$
           c_k = \min_{k_1 + k_2 =k} (a_{k_1} + b_{k_2}).
           $$
           
           \begin{prop} For any $f, h \in Ham(M,\om)$, 
           $$
           \| f \circ h \|_T^+     \leq      \|f \|_T^+   \star   \|h \|_T^+.
           $$
           \end{prop}
           
             Let $f \in \Ham(M,\omega)$, and consider by abuse of notation $\| f \|_T^+ : \NN \to \RR$ the application assigning to each integer $n$ the coefficient $\|f\|_n^+$. Then $\| f \|_T^+$ is evidently a non-negative and non-increasing application.
             
             The following question is the principal conjecture concerning the behaviour of $g$-areas.
             
             \begin{conjecture} For all $f \in \Ham(M,\omega)$, the application $\| f \|_T^+ : \NN \to \RR$ is convex, that is to say: for all $g > 0$,
             $$ 
            \| f \|_{g+1}^+ - \| f \|_{g}^+  \le  \| f \|_g^+ - \| f \|_{g-1}^+
                $$
             \end{conjecture}

        Before giving the proofs of our theorems below, let us examine the infinitesimal version of these results.
    
       We have seen that, since the group of Hamiltonian diffeomorphisms of a closed symplectic manifold $(M,\omega)$ is simple, any Hamiltonian diffeomorphism  of $M$ can be written as a product of a finite number of commutators. Moreover we show in the above theorem that the $g_+$-area of a Hamiltonian diffeomorphism $f$ can be regarded as the obstruction to writing $f$ as a product of $g$ commutators. It is natural to ask whether there exist {\it infinitesimal versions} of these results and constructions.  First of all note that, by a well known theorem due to Lichnerowicz (and also to Calabi and  Rosenfeld) (see Theorem 1.4.3 in \cite{Ba}) the Lie algebra of Hamiltonian vector fields is perfect, so  any Hamiltonian vector field can be written as a linear combination of Lie brackets of Hamiltonian vector fields. Therefore it is natural to ask
\\ \\       
{\bf Question:} {\it Can one define for every $k\in\N$ an invariant for  Hamiltonian vector fields which can be interpreted as  the  obstruction to writing the given vector field as a linear combination of $k$ Lie brackets of Hamiltonian vector fields?}
\\

It is convenient to replace Hamiltonian vector fields by smooth functions  (modulo constants) using the correspondence $f\mapsto X_f$, and the Lie bracket by the  Poisson bracket. The two questions above have obvious reformulations in terms of Hamiltonians.

We believe that the first step in studying these questions is 
to describe explicitly the cone of {\it infinitesimal commutators} $\{\{f,g\}|\ f,\ g\in {\Cal C}^\infty(M,\R)\}$ in ${\Cal C}^\infty(M,\R)$, but we realized that this question is already quite difficult. In Sect. 3 we will make progress in this direction which, although modest, shows that the problem is interesting and difficult:  we will fix a Morse function $f$ (satisfying a weak genericity condition)  on a closed symplectic surface $(M,\omega)$ and we will describe explicitly the space of functions $u\in {\Cal C}^\infty(M,\R)$ which can be written under the form $u=\{f,g\}$ (or, equivalently, $u=X_f(g)$, where $X_f$ is the Hamiltonian vector field of $f$) with $g\in {\Cal C}^\infty(M,\R)$.

The referee kindly informed us that  our problem is often called in the literature "the cohomological equation for flows"  and there exists an ample literature dedicated to this  problem for an interesting  class of  Hamiltonian vector fields on higher dimensional manifolds, for instance for Hamiltonian  fields defining  Anosov flows. The so called Livsic theory (see \cite{Li}) deals with this problem. A special example of such a  flow is   the geodesic flow  on a  hyperbolic surface, which has been investigated by many authors  (see \cite{CEG}, \cite{FF}, \cite{LL}, \cite{LLMM}).   Further important results dedicated to  Livsic's cohomological equation concern the  ''locally
Hamiltonian flows'', or area-preserving flows  (on surfaces of higher genus) with canonical saddle-like singularities and their return maps
(interval exchange transformations)  \cite{Fo}, \cite{MMY}.

\medskip
\noindent
{\bf Acknowledgements.} We are very grateful to Michael Entov for pointing out the relation between his notion of size and the invariants introduced in this paper.  We are also grateful to the referee for his pertinent comments and interesting bibliographic references.

             \section{Proof of Theorem~\ref{thm:theorem2}}  
             
             We begin by constructing, for $f = [\phi, \psi] \in \Ham(M,\omega)$ and $\eps >0$ given, a ruled symplectic fibration $(M, \om) \hookrightarrow (P_f, \Om) \to \Si_1$ with fiber $(M, \om)$ and base the punctured torus, that satisfies:
             
             1) the monodromy of $P_f$ round the boundary of $\Sigma_1$ is equal to $f$, and
             
             2) $A(P_f, \Om) \leq \eps$
             
             where $A(P_f, \Om)$ is the area of $P_f$ defined as the quotient of the volume of $ (P_f, \Om)$ by the volume of $(M, \om)$.
             
             To achieve this, let us take two copies $A_1, A_2$ of the annulus $\RR/3\ZZ \times [0,1].$ Let us glue $A_1$ to $A_2$ by the rotation
           $$
           R: [0,1] \times [0,1] ( \subset A_1)  \to   [0,1] \times [0,1] ( \subset A_2)
           $$
           of angle equal to $\pi/2$. The space $A = A_1 \cup_R A_2$ is a thickening of the $1$-squeletton, homeomorphic to the punctured torus $\Si_1$. Consider now the direct product $(A, \sigma) \times (M, \om)$ where $\sigma$ is the obvious induced area form from  the two copies of $\RR/3\ZZ \times [0,1]$. Cut $A \times M$ over the segment $B_1 = \{2\} \times [0,1] \subset A_1$ and glue $B_1^- = \{2^-\} \times [0,1] \times M$ to $B_1^+ = \{2^+\} \times [0,1] \times M$ using $\phi$. Similarly, cut $A \times M$ over the segment $B_2 = \{2\} \times [0,1] \subset A_2$ and glue $B_2^- = \{2^-\} \times [0,1] \times M$ to $B_2^+ = \{2^+\} \times [0,1] \times M$  using $\psi$.
           
\begin{figure}%[h]
\centering
\scalebox{0.6}
{\includegraphics{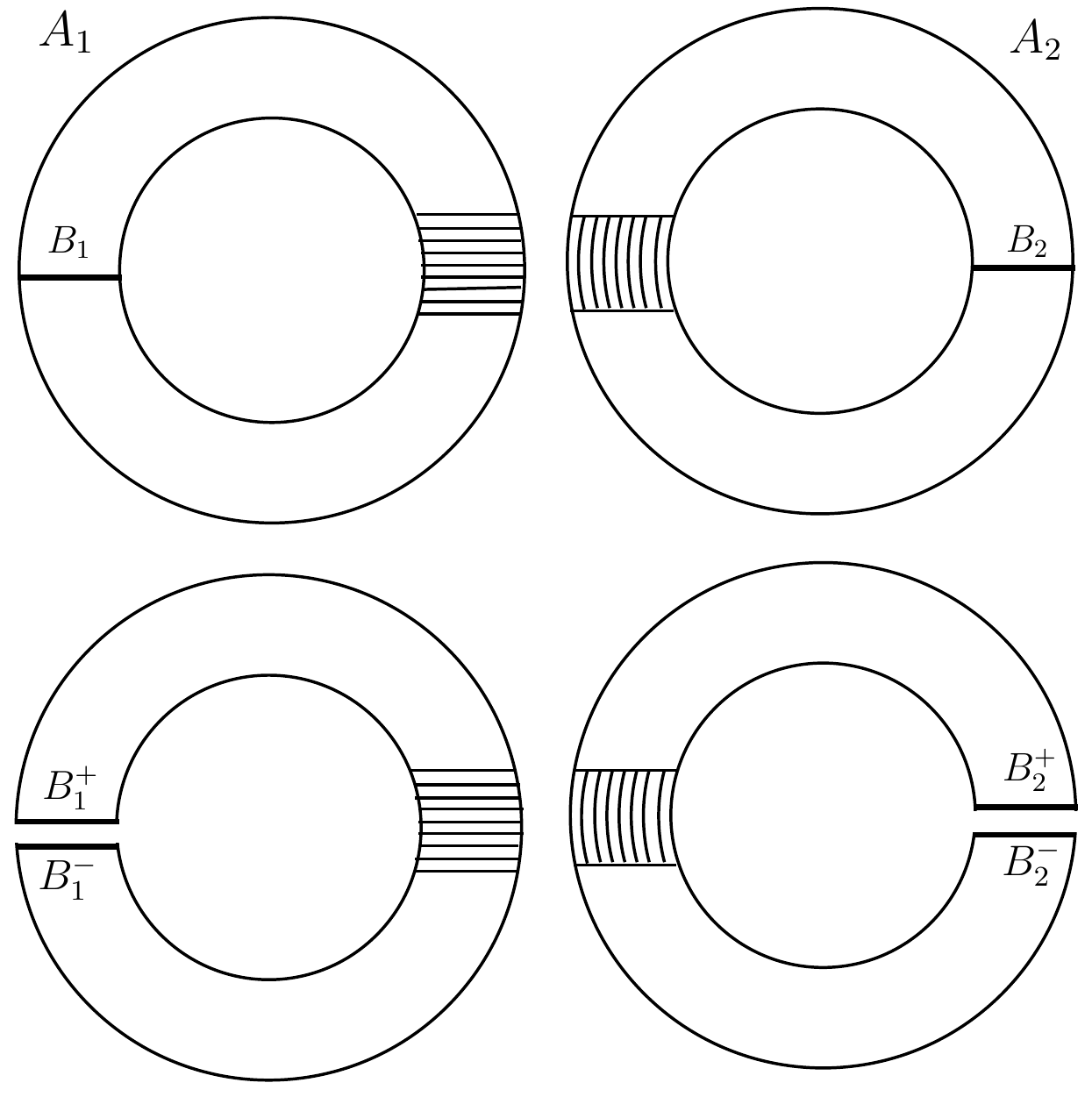}}
\label{base-new}
\end{figure}

           It is clear that the monodromy of the resulting fibration $(M, \om) \hookrightarrow (P_f, \Om) \to A \simeq \Si_1$ over the boundary  $\partial A$ is $\psi^{-1} \circ \phi^{-1} \circ \psi \circ \phi = [\phi, \psi]$. On the other hand, the area $A(P, \Om)$ can be chosen as small as one wishes.
           
           Now, let a product of $k$ commutators $f = [\phi_1, \psi_1] \ldots [\phi_k, \psi_k]$ be given. Let $p_0$ be the base point on the boundary of the base and for each $1 \leq i \leq k$ let $P_{f_i =   [\phi_i, \psi_i]} \to \Si_{1,i}$ be the symplectic fibration constructed above with arbitrary small area and monodromy equal to $f_i =   [\phi_i, \psi_i]$. We glue together the $k$ fibrations by first glueing the base near the base point and then by identifying the corresponding fibers over a small neighbourhood of $p_0$ where the fibrations are trivial. This construction shows that each handle can be used to produce, and therefore to kill,  a commutator.
           
           To prove Theorem\ref{thm:theorem2}, we must find for each $f \in \Ham(M,\omega), h \in \Cc_k$ and $\eps > 0$, a fibration $(M, \om) \hookrightarrow (P, \Om) \to \Si_k$ whose monodromy is $f$ and whose area is bounded above by $\|f \circ h^{-1}\|_0^+ + \eps$. Let $(M, \om) \hookrightarrow (P', \Om) \to \Si_0$ be a fibration whose area is bounded above by  $\|f \circ h^{-1}\|_0^+ + \eps/2$ and whose monodromy is equal to $f \circ h^{-1}$. This exists by the definition of the $0$-distance between $f$ and $h$. Since $h$ belongs to $\Cc_k$, there is also a fibration $(M, \om) \hookrightarrow (P'', \Om) \to \Si_k$ of area bounded above by $\eps/2$ and monodromy $h$. The connected sum of both fibrations $P'$ and $P''$ near the base point $p_0$ of the boundary gives a fibration $P$ over $\Si_k$ of area bounded above by $\|f \circ h^{-1}\|_0^+ + \eps$ with monodromy $f \circ h^{-1} \circ h = f$. This completes the proofs of Theorems~\ref{thm:theorem1} and \ref{thm:theorem2}.
           
\subsection{The relation with Entov's article} \label{ss:entov}

Let us fix an area form $\sigma$ of total area 1 on the surface $\Sigma_g$. 

Since a surface with boundary is homotopy equivalent to its 1-skeleton and the structure group $\Ham (M,\omega)$ is connected, all our Hamiltonian fibrations are topologically trivial. This means that the (relative, modulo the boundaries) cohomology classes of the forms are related as follows:

$$
[\Omega] = [\pi_1^* \omega] + \tau [\pi_2^* \sigma]
$$
for some real positive $\tau$, where $\pi_1$ and $\pi_2$ are the projections of $P$ to $M$ and $\Sigma_g$ respectively under a trivialization of the fibration. 
Thus 

$$
vol(P,\Omega) = \int_P [\Omega]^{n+1}= \tau \int_P [\omega]^n [\sigma],    
$$
$$
vol(M,\omega) = \int_M [\omega]^n
$$
and the ratio $vol(P,\Omega)/vol(M,\omega)$ is $\tau$.

In other words, our $g_+$ area is  the infimum of all positive $\tau$ for which one can find a ruled symplectic form $\Omega$ on P representing the class $[\omega]+\tau [\sigma]$ and having the prescribed holonomy over the boundary.

Now $1/g_+$ is more or less the notion of $size_g$ appearing in Entov's paper \cite{E} in Section 5 (its analog for Hamiltonian fibrations over $S^2$ had appeared before in Polterovich's papers cited there). One of the main differences with our setup is that we work with the group $\Ham$ while in Entov's paper, $size_g$ is defined for elements of the universal cover of $\Ham$. 

Then Proposition 5.0.9 in Entov's paper (with $\ell=1$), translated to our setup and proved by the same methods as in Polterovich's previous papers for the case of fibrations over $S^2$, would say that: \\
\\
$g_+ - area (f) = 1/size_g (f) \leq $ {\it  the Hofer distance from $f $ to the set of elements of $\Ham(M,\omega)$ which are products of at most $g$ commutators.}\\

This is weaker than our Theorem 1.2 : indeed in our Theorem, we distinguish between the positive and negative Hofer norm, yielding a finer result. In fact,  our proof is completely different from Entov-Polterovich.

\section{The space of infinitesimal commutators on a closed symplectic surface}

Let $(M,\omega)$ be a closed symplectic manifold. Our goal is to determine explicitly the space
of infinitesimal commutators
$${\cal C}_\omega:=\big\{\{f,g\}\ \vline \ f,\ g\in {\cal C}^\infty(M,\R)\big\}\ .
$$
Recall that $\{f,g\}:=\iota_{X_f}dg$,  where $X_f$   is the Hamiltonian vector field     associated with $f$. The first step is to compute, for a fixed smooth function $f\in {\cal C}^\infty(M,\R)$ the space
$${\cal C}_{\omega,f}:=\big\{\{f,g\}\ \vline  \ g\in {\cal C}^\infty(M,\R)\big\}\ .
$$
A pair $s=(\sigma,\xi)$ consisting of a connected 1-dimensional manifold and a nowhere vanishing vector field $\xi\in{\cal X}(\sigma)$ of $\sigma$ will be called {\it framed 1-manifold}. A framed 1-manifold $s=(\sigma,\xi)$ will be called {\it framed circle} if $\sigma$ is a circle; in this case $\xi$ is induced by a periodic map $\lambda:\R\to \sigma$, which is well defined up to reparametrisation given by a translation, so the period $\tau(s)\in(0,\infty)$ of $\lambda$ is well defined (it depends only on the pair $s$).    Such a parametrisation $\lambda$ will be called compatible. For a compact framed 1-manifold with non-empty boundary, a compatible parametrization will be a diffeomorphism $\lambda:[t,  t+\tau(s)]\to\sigma$ inducing $\xi$.

Let $s=(\sigma,\xi)$ be a compact (closed or compact with boundary) framed 1-manifold, and denote by $\mu_\xi$   the 1-form on $\sigma$ defined by $\langle \mu_\xi,\xi\rangle\equiv 1$.  For a  continuous function  $f$  on $\sigma$  we put 
$$\int _s f:=\int_\sigma f \mu_\xi=\int_t^{t+\tau(s)} (f\circ\lambda)\lambda^*(\mu_\xi)=\int_t^{t+\tau(s)} f(\lambda(t)) d t\ ,
$$
where  $\lambda$ is a compatible parametrisation of $s$.

Note that any  integral circle (non-constant periodic orbit) of a vector field $X$ on a manifold can be regarded naturally as a (closed) framed circle.

\begin{re} Let  $(M,\omega)$ be a symplectic manifold, $f$, $g\in{\cal C}^\infty(M,\R)$ smooth functions on $M$. Then
\begin{enumerate}
\item $\{f,g\}(p)=0$ for any critical point $p$ of $f$.
\item $\int_s \{f,g\}=0$
for every  integral circle of  the Hamiltonian vector field $X_f$.
\end{enumerate}
\end{re}
Therefore
\begin{equation}\label{incl}{\cal C}_{\omega,f}\subset \left\{u\in{\cal C}^\infty(X,\R)\vline\begin{array}{c} 
\resto{u}{{\rm Crit}(f)}\equiv 0\\  \int_s u=0\ \forall s \hbox{  integral circle   of }X_f\end{array}\right\}.
\end{equation}

\begin{thry} \label{Th} The inclusion (\ref{incl}) is an equality when $M$ is a closed surface and $f$ is a Morse function with the property that any connected component of a level set of $f$ contains at most one index 1 - critical point.
\end{thry}
\begin{proof}

Let $u\in{\cal C}^\infty(M,\R)$  such that 
\begin{equation}\label{cycle}\int_s u=0\ \forall s \hbox{   integral circle   of }X_f\ .
\end{equation}
  We will show that the differential equation $\{f,g\}=u$  has a smooth solution $g\in{\cal C}^\infty(M,\R)$.
\vspace{2mm}\\
{\bf Step 1}: Local solvability.
\\

The interesting part of the proof is to show that condition (\ref{cycle}) which has a global character (it concerns the behavior of $u$ on certain curves contained in $M$) implies the local conditions which are necessary  to assure the local solvability of our equation around the critical points. Whereas this is not surprising for critical points of index 0 and 2 (because  there exists integral circles which are arbitrary  close to such a point), for critical points of index 1 passing from  (\ref{cycle}) to the needed local solvability conditions is not obvious at all.  The argument is based on  the solvability Theorem \ref{many} stated at the end of the section and a geometric remark.\\

Suppose that $p_0$ is a critical point of $f$, and let $p_0\in U\textmap{h} V\subset\R^2$ be a Morse chart around $p_0$ for $f$, i.e. a chart  such that $h(p_0)=0$, $\resto{f}{U}(x,y)=\frac{1}{2}(\pm x^2\pm y^2)$, and write $\resto{\omega}{U}=  a dx\wedge dy$. We may suppose that our chart is compatible with the symplectic orientation, so $a$ is  positive on $U$. Therefore
$$X_f= \frac{1}{a}\left[\frac{\partial f}{\partial y}\frac{\partial }{\partial x}-\frac{\partial f}{\partial x}\frac{\partial}{\partial y}\right]\ ,\ \iota_{X_f} dg=\frac{1}{a}\left[\frac{\partial f}{\partial y}\frac{\partial g }{\partial x}-\frac{\partial f}{\partial x}\frac{\partial g}{\partial y} \right]\ .
$$
{\ }\\
{\bf A}. Suppose  $p_0$ has index 0,  i.e. it is a local minimum and $f(x,y)=\frac{1}{2}(x^2+y^2)$, $\resto{\omega}{U}= a dx\wedge dy$. In this case $\{f,g\}=\frac{1}{a}(y\frac{\partial g }{\partial x}-x\frac{\partial g}{\partial y})$ and the condition $\{f,g\}=u$
becomes $P_Xg =-au$, where $X=-y\frac{\partial}{\partial x}+x\frac{\partial}{\partial y}$ is the vector field whose associated first order differential equation is dealt with   in Theorem  \ref{newtheorem} below.  The same arguments applies around index 2 -  critical points. 
\\ \\
{\bf B}.  Suppose now that $p_0$ has index 1. Using a linear change of coordinates we may suppose that $\resto{f}{U}(x,y)=xy$, so $X_f=\frac{1}{a}(x\frac{\partial}{\partial x}-y\frac{\partial}{\partial y})$, which, up to the factor $\frac{1}{a}$ is just the vector field whose associated differential equation is studied in Theorem \ref{many} below.

We will prove that the equation $P_{X_f} g=u$ is locally solvable around $p_0$. Let $Z$ be a vector field on $M\setminus\mathrm{Crit}(f)$ such that $df(Z)\equiv 1$. Let $(\Psi_t)_{t\in\R}$ be the corresponding 1-parameter group, every $\Psi_t$ being considered as usual on its maximal domain.

Let $c_0$ be the connected component of the level set $f^{-1}(f(p_0))$ 
containing $p_0$. Taking into account our assumption about $f$, $p_0$ is the only critical point belonging to $c_0$. Let $\gamma_0\subset c_0$ be a singular circle containing $p_0$, i.e. the closure of a non-compact integral line of $X_f$ converging in both directions to $p_0$. Let  $p_0'$, $p_0''\in \gamma_0$ be two points close to $p_0$ belonging to the two branches which intersect transversally in $p_0$. These points cut $\gamma_0$ in two segments $\mu_0$, $\nu_0$, where the notations were chosen such that $p_0\in\nu_0$.

\begin{figure}%[h]
\centering
\scalebox{0.8}
{\includegraphics{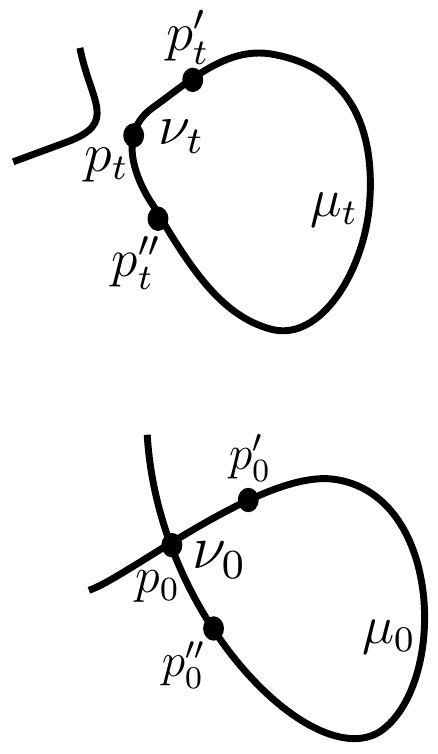}}
\label{critical}
\end{figure}
%\vspace{-4mm}

Consider the image $\Psi_t(\gamma_0\setminus\{p_0\}$) for $t\in(-\epsilon,\epsilon)$. The point is that, either for positive or for negative $t$, the image $\Psi_t(\gamma_0\setminus\{p_0\})$ can be written as $\gamma_t\setminus\{p_t\}$, where $\gamma_t$ is a smooth circle (containing no critical points) converging to $\gamma_0$,  and $(p_t)$ is a smooth path converging to $p_0$   as $t\to 0$. Suppose that this occurs for positive $t$. 
We   put $p'_t:=\Psi_t(p'_0)$, $p''_t:=\Psi_t(p_0'')$, $\mu_t:=\Psi_t(\mu_0)$ for $t\in(-\epsilon,\epsilon)$, and $\nu_t:=\overline{\gamma_t\setminus\mu_t} $ for $t\in(0,\epsilon)$. All the segments $\mu_t$, $\nu_t$ are naturally framed by the field $X_f$.

The argument is now very simple: since $\gamma_t$ is an integral circle of $X_f$  our hypothesis implies $0=\int_{\gamma_t} u=\int_{\mu_t} u+\int_{\nu_t} u$ for any $t\in(0,\epsilon)$, so  for positive $t$ one has $\int_{\nu_t} u=-\int_{\mu_t} u$, which extends smoothly for   $t\leq 0$, because the segment $\mu_t$ does. This shows that one of the functions $\phi_u^{++}$,  $\phi_u^{--}$, $\phi_u^{+-}$, $\phi_u^{-+}$ appearing in Theorem \ref{many} below extends smoothly at 0, which, by this theorem, implies that the equation $P_{X_f}g=u$ is locally solvable around $p_0$.

\vspace{1mm} {\ }\\
{\bf Step 2}: Global solvability.  Let ${\cal K}$ be the sheaf of germs of (locally defined) smooth functions  $\kappa$ satisfying the equation $P_{X_f} \kappa=0$ associated with the vector field $X_f$.  According to the first step, there exists an open cover $(U_i)_{i\in I}$ of $M$ and local solutions $g_i:U_i\to \R$ of the equation $P_{X_f} g=u$. The system $\kappa_{ij}=g_j- g_i\in{\cal K}(U_i\cap U_j)$ defines a 1-cohomology class $\oo(u)\in H^1(M,{\cal K})$, which is the obstruction to the construction of a global solution.

According to Lemma \ref{coho} below, the space $H^1(M,{\cal K})$ can be embedded in the space of smooth functions ${\cal H}_1\to\R$ which are fibrewise linear, where ${\cal H}_1$ is the union of the 1-homology groups of the connected components of the level sets of $f$, endowed with a natural topology and differentiable structure (see the construction below). For an integral circle $s$ of $X_f$ it is easy to prove that $\oo([s])=\int_s u$. Therefore our assumption implies that $\oo$ vanishes on the 1-homology of every smooth 1-dimensional component of a level set. If $c_0$ is a singular 1-dimensional component, it will contain only one index 1 - critical point, and $H_1(c_0)\simeq\Z^2$ is generated by the fundamental classes of two singular circles. But these circles are limits (in ${\cal H}_1$) of smooth integral circles, so our assumption implies $\oo=0$.

\end{proof}{\ }\vspace{-10mm}\\
{\it The cohomology of ${\cal K}$.}
\vspace{-2mm}\\ 

Let ${\cal K}$ be the sheaf of germs of (locally defined) smooth functions  $\kappa$ satisfying the equation $P_{X_f}\kappa=0$ on $M$  and 
 let ${\cal C}$ be the space of connected components of level sets of $f$, endowed with the quotient topology. This space comes with two obvious continuous surjections 
$$M\textmap{\pi}{\cal C}\textmap{\fg} [a,b]:=f(M)\ .$$
 ${\cal C}$ has  naturally the structure of a connected 1-dimensional CW complex, whose 0-cells correspond to the critical points of $f$ and whose 1-cells are mapped homeomorphically on sub-intervals of $[a,b]$ connecting two critical values. We  will compute the cohomology of the sheaf ${\cal K}$ on $M$ using the Leray spectral sequence associated with the projection $\pi:M\to {\cal C}$ (see \cite{G} Theorem 4.17.1, p. 201).  
\\

Every point $t\in[a,b]$ has a neighborhood $J_t\subset[a,b]$ such that
the inclusion $f^{-1}(t)\hookrightarrow f^{-1}(J_t)$ is a homotopy equivalence. 

The disjoint union
 ${\cal H}_1:= \cup_{c\in{\cal C}} H_1(c)
$ 
comes with a natural surjection  $\chi: {\cal H}_1\to  {\cal C}$ and has a natural structure of 1-manifold with boundary such that $\fg\circ\chi:{\cal H}_1\to[a,b]$ becomes an immersion. A section $[a,b]\supset J\ni t\mapsto h(t)$ is smooth with respect to this structure if for every $t\in J$ and $t'\in J_t\cap J$ the classes $h(t')$, $h(t)$ coincide in $H_1(f^{-1}(J_t))$.

Let $U\subset {\cal C}$ open. A continuous function $\beta:U\to\R$ will be called {\it smooth} if for every $c\in U$ and section $\gamma:V\to U\subset{\cal C}$ of $\fg$ defined on an open neighborhood $V$ of $\fg(c)$ in $[a,b]$, the composition $\beta\circ\gamma$ is smooth  on $V$.

\begin{lm} \label{coho} With the assumptions and the notations above:  
\begin{enumerate}
\item For an open set $U\subset {\cal C}$  the space $R^0(\pi_*)({\cal K})(U)$  can be embedded in the space of smooth functions on $U$. This sheaf if fine.
\item  For an open set $U\subset {\cal C}$  the space $R^1(\pi_*)({\cal K})(U)$ can be embedded in the space of smooth functions on $\chi^{-1}(U)\subset{\cal H}_1$ which are fibrewise group-morphisms.
\item The sheaves $R^i(\pi_*)({\cal K})(U)$ vanish for $i>1$.
\item The cohomology group $H^1(M,{\cal K})$ can be embedded in the space of smooth functions ${\cal H}_1\to\R$ which are fibrewise linear.
\end{enumerate}
\end{lm}
 
\begin{proof}  The proof is straightforward, so the details will be omitted. The only difficulty is to describe explicitly the stalk $R^i(\pi_*)({\cal K})_{c_0}$ at a point $c_0\in{\Cal C}$, which contains an index 1 - critical point $p_0$.  Using the general theory of Leray spectral sequences (see   \cite{G}   p. 201) we see that $R^i(\pi_*)({\cal K})_{c_0}=H^i(c_0, {\cal K}^0)$,
where ${\cal K}^0:=\resto{\cal K}{c_0}$ is the restriction of the sheaf ${\cal K}$ to $c_0$, i.e. the sheaf whose stalk at a point $p\in c_0$ coincides with the stalk  ${\cal K}_{p}$.  The point is that $c_0$ is a CW complex, and ${\cal K}^0$ a {\it cellular sheaf} in the sense of \cite{T}, i.e. a sheaf whose restriction to every cell is constant. The obvious CW structure of $c_0$ has a 0-cell, namely $\{p_0\}$ and  two 1-cells, denoted   $\mu_0$, $\nu_0$ (the  connected components of $c_0\setminus\{p_0\}$).  According to the main result of \cite{T}, the cohomology of a cellular sheaf can be computed using a cochain complex constructed in a similar way  as the usual cellular cochain complex which computes the singular cohomology.  Using this result we obtain (2) and (3).
The last statement is obtained  using the Leray spectral sequence associated with the projection $\pi:M\to{\cal C}$
\end{proof}  {\ }\vspace{-3mm}\\
{\it Local solvability theorems:}\vspace{-3mm}\\

We could not find in the literature the local solvability results used above, so we state them here for completeness, indicateing only briefly the method of proof:

 \begin{thry} \label{newtheorem} Let $X=-y\frac{\partial}{\partial x}+x\frac{\partial}{\partial y}$ and $r\in(0,\infty]$. A function $u\in {\cal C}^\infty(B_r,\R)$ belongs to $\im(P_X:{\cal C}^\infty(B_r,\R)\to {\cal C}^\infty(B_r,\R))$ if and only if 
\begin{equation}\label{int}\int_{0}^{2\pi} u(\rho\cos(t),\rho\sin(t))dt=0\ \forall \rho\in [0,r)\ .
\end{equation}
\end{thry}

The condition is obviously necessary. In order to prove that it is also sufficient, we use polar coordinates and we solve the equation $\frac{\partial}{\partial \theta} g=u$ on the circles $C(0,\rho)$ ($0\leq\rho < r$) with the condition 
$$\int_0^{2\pi} g(\rho (\cos(t), \sin(t)))dt = 0\ \ \forall \rho\in[0,r) \ .
$$
Explicit computations show that the function $g:B(0,r)\to \R$ is smooth and solves the equation $Xg=u$.
\\

For the vector field   $X=x\frac\partial{\partial x}-y\frac\partial{\partial y}$ the solvability problem has a completely different answer. This germ of this vector field at 0 is hyperbolic, so the corresponding {\it  flat problem} (see \cite{Rou}, sections I.2.2, II.2.1)) is always solvable. Using  the method explained in the introduction of \cite{Rou}  (page 15),  we see that the equation $X(g)=u$ has a ${\cal C}^\infty$ solution around 0 if and only if it has a formal solution.   The first order operator $P_X$ associated with $X$ commutes with the hyperbolic order 2-operator $Q:=\frac{\partial^2}{\partial x\partial y}$, and we see easily that a formal solution exists if and only if
\begin{equation}\label{solv}
(Q^nu)(0)=0\ ,\ \forall n\in\N\ .
\end{equation}
Therefore (\ref{solv}) is  the ${\cal C}^\infty$ solvability condition around 0 for the equation $X(g)=u$.
In the proof of Theorem  \ref{Th}  we need a simple consequence of this criterion.
Let $u\in {\cal C}^\infty(\R^2,\R)$. We define the functions 
$\phi_u^{++}\ ,\ \phi_u^{--}:(0,1)\to\R\ ,\ \phi_u^{+-}\ ,\ \phi^{-+}_u:(-1,0)\to\R$
 by
$$\phi_u^{++}(\rho)=\int_{\ln(\rho)}^0 u(e^t,\rho e^{-t}) dt\ ,\ \phi_u^{--}(\rho)=\int_{\ln(\rho)}^0 u(-e^t,-\rho e^{-t}) dt\ ,
$$
$$\phi_u^{+-}(\rho)=\int_{\ln(-\rho)}^0 u(e^t,\rho e^{-t}) dt\ ,\ \phi_u^{-+}(\rho)=\int_{\ln(-\rho)}^0 u(-e^t,-\rho e^{-t}) dt\ .
$$

Note that,  for $\rho\in(0,1)$ (respectively $\rho\in(-1,0)$)   $\phi_u^{++}(\rho)$ and $\phi_u^{--}(\rho)$ (respectively $\phi_u^{+-}(\rho)$ and $\phi_u^{-+}(\rho)$)  are defined by integrating $u$ along the  two arcs obtained by intersecting the  hyperbola $H_\rho:=\{(x,y)\in\R^2|\ xy=\rho\}$ (parameterized as integral curve of $X$) with the square $[-1,1]\times[-1,1]$. Since these   parameterized arcs do not have a limit as $\rho\to 0$,  we cannot expect the four functions to extend continuously at 0. But the geometric interpretation of the four functions (as integrals along integral curves of $X$) shows that, if $u=P_Xg$ for a smooth function $g$ defined around this square, we will have
\begin{equation}\label{diff}\phi_u^{++}(\rho)=g(1,\rho)-g(\rho,1)\ ,\ \phi_u^{--}(\rho)=g(-1,-\rho)-g(-\rho,-1)\ ,\ \rho\in(0,1)\ ,
\end{equation}
$$\phi_u^{+-}(\rho)=g(1,\rho)-g(-\rho,-1)\ ,\ \phi^{-+}_u(\rho)=g(-1,-\rho)-g(\rho,1)\ , \ \rho\in(-1,0)\ .
$$
Therefore, if $u=P_Xg$ for a smooth function $g$ defined around $[-1,1]\times[-1,1]$, then  the functions $\phi_u^{++}$, $\phi_u^{--}$, $\phi_u^{+-}$,   $\phi^{-+}_u$ extend  ${\cal C}^\infty$  at 0. The same is true when $u$ can be written as $P_Xg$ on a smaller square $[-r,r]\times[-r,r]$  with $r>0$, so when the equation $P_X g=u$  is locally ${\cal C}^\infty$-solvable around the origin.  The following theorem states that the converse is also true.

\begin{thry}\label{many} Let $u\in{\cal C}^\infty(\R^2,\R)$. The following conditions are equivalent:  
\begin{enumerate}

\item The functions $\phi_u^{++}$, $\phi_u^{--}$, $\phi_u^{+-}$,   $\phi^{-+}_u$ extend  ${\cal C}^\infty$  at 0.
\item One of the four functions $\phi_u^{++}$, $\phi_u^{--}$, $\phi_u^{+-}$,   $\phi^{-+}_u$ extends  ${\cal C}^\infty$  at 0. 
\item The derivatives of $u\in{\cal C}^\infty(\R^2,\R)$ at 0 satisfy  (\ref{solv})
\item The equation $P_X g=u$ has a ${\cal C}^\infty$-solution around 0.
\end{enumerate}
\end{thry}

  \end{document}